\documentclass[letterpaper, 11pt,  reqno]{amsart}

\usepackage{amsmath,amssymb,amscd,amsthm,amsxtra, esint}

\usepackage{hyperref} 
\allowdisplaybreaks[2]

\newtheorem{theorem}{Theorem} [section]

\newtheorem{lemma}[theorem]{Lemma}
\newtheorem{proposition}[theorem]{Proposition}
\newtheorem{remark}[theorem]{Remark}

\newtheorem{definition}[theorem]{Definition}
\newtheorem{corollary}[theorem]{Corollary}

\DeclareMathOperator*{\intt}{\int}

\DeclareMathOperator*{\supp}{supp}

\newcommand{\I}{\hspace{0.5mm}\text{I}\hspace{0.5mm}}
\newcommand{\noi}{\noindent}
\newcommand{\Z}{\mathbb{Z}}
\newcommand{\R}{\mathbb{R}}

\newcommand{\T}{\mathbb{T}}

\let\Re=\undefined\DeclareMathOperator*{\Re}{Re}

\newcommand{\F}{\mathcal{F}}

\newcommand{\al}{\alpha}

\newcommand{\dl}{\delta}

\newcommand{\Dl}{\Delta}
\newcommand{\eps}{\varepsilon}

\newcommand{\ld}{\lambda}

\newcommand{\s}{\sigma}
\newcommand{\Si}{\Sigma}
\newcommand{\ft}{\widehat}

\newcommand{\wt}{\widetilde}
\newcommand{\cj}{\overline}

\newcommand{\dt}{\partial_t}

\renewcommand{\o}{\omega}

\newcommand{\les}{\lesssim}
\newcommand{\ges}{\gtrsim}

\newcommand{\jb}[1]
{\langle #1 \rangle}

\newcommand{\ind}{\mathbf 1}

\renewcommand{\S}{\mathcal{S}}

\newcommand{\M}{\mathcal{M}}

\newcommand{\N}{\mathbb{N}}

\newcommand{\NN}{\mathcal{N}}

\renewcommand{\I}{\mathcal{I}}

\newcommand{\TT}{\mathcal{T}}

\newcommand{\BT}{{\bf T}}%{\mathfrak{T}}

\usepackage{tikz}

\usetikzlibrary{shapes.misc}
\usetikzlibrary{shapes.symbols}
\usetikzlibrary{shapes.geometric}
\tikzset{
	dot/.style={circle,fill=black,draw=black,inner sep=0pt,minimum size=0.5mm},
	>=stealth,
	}
\tikzset{
	ddot/.style={circle,fill=white,draw=black,inner sep=0pt,minimum size=0.8mm},
	>=stealth,
	}

\tikzset{decision/.style={ % requires library shapes.geometric
        draw,
        diamond,
        aspect=1.5
    }}

\tikzset{dia2/.style
={diamond,fill=white,draw=black,inner sep=0pt,minimum size=1mm},
	>=stealth,
	}

\tikzset{dia/.style
={star,fill=black,draw=black,inner sep=0pt,minimum size=1mm},
	>=stealth,
	}

\makeatletter
\def\DeclareSymbol#1#2#3{\expandafter\gdef\csname MH@symb@#1\endcsname{\tikz[baseline=#2,scale=0.15]{#3}}}
\def\<#1>{\csname MH@symb@#1\endcsname}
\makeatother

\DeclareSymbol{X}{-2.4}{\node[dot] {};}
\DeclareSymbol{1}{0}{\draw[white] (-.4,0) -- (.4,0); \draw (0,0)  -- (0,1.2) node[dot] {};}
\DeclareSymbol{2}{0}{\draw (-0.5,1.2) node[dot] {} -- (0,0) -- (0.5,1.2) node[dot] {};}

\DeclareSymbol{1}{-2.7}
 {
  \draw (0,0) node[dot]{};}

\DeclareSymbol{1'}{-2.7}
 {\draw (0,0) node[ddot]{};}

\DeclareSymbol{1''}{-2.7}
 {\draw (0,0) node[dia]{};}

\DeclareSymbol{3}{0}
 {\draw (0,0) node[dot]{} -- (0,1.2) node[ddot] {}; 
 \draw (-1.2,0) node[dot] {} -- (0,1.2)node[ddot] {} -- (1.2,0) node[dot] {};}

\DeclareSymbol{3'}{0}
 {\draw (0,0) node[dia]{} -- (0,1.2) node[ddot] {}; 
 \draw (-1.2,0) node[dia] {} -- (0,1.2)node[ddot] {} -- (1.2,0) node[dia] {};}

\DeclareSymbol{31}{-3}{
\draw (0,-1)node[dot] {} -- (0,0) node[ddot] {}-- (0.9, -1) node[dot] {}; 
\draw (0,-1)node[dot] {} -- (0,0) node[ddot] {}-- (-0.9, -1) node[dot] {}; 
\draw (0,0)node[ddot] {} -- (1.3,1) node[ddot] {}-- (1.3, 0) node[dot] {}; 
\draw  (1.3,1) node[ddot] {}-- (2.6, 0) node[dot] {}; 
}

\DeclareSymbol{32}{-3}{
\draw (0,-1)node[dot] {} -- (0,0) node[ddot] {}-- (0.9, -1) node[dot] {}; 
\draw (0,-1)node[dot] {} -- (0,0) node[ddot] {}-- (-0.9, -1) node[dot] {}; 
\draw (0,0)node[ddot] {} -- (0,1) node[ddot] {}-- (0.9, 0) node[dot] {}; 
\draw (0,0)node[ddot] {} -- (0,1) node[ddot] {}-- (-0.9, 0) node[dot] {}; 
}

\DeclareSymbol{33}{-3}{
\draw (2.6,-1)node[dot] {} -- (2.6,0) node[ddot] {}-- (3.5, -1) node[dot] {}; 
\draw (2.6,-1)node[dot] {} -- (2.6,0) node[ddot] {}-- (1.7, -1) node[dot] {}; 
\draw (0,0)node[dot] {} -- (1.3,1) node[ddot] {}-- (1.3, 0) node[dot] {}; 
\draw  (1.3,1) node[ddot] {}-- (2.6, 0) node[ddot] {}; 
}

\DeclareSymbol{3''}{0}
 {\draw (0,0) node[dia]{} -- (0,1.2) node[ddot] {}; 
 \draw (-1.2,0) node[dot] {} -- (0,1.2)node[ddot] {} -- (1.2,0) node[dia] {};}

\DeclareSymbol{3'''}{0}
 {\draw (0,0) node[dot]{} -- (0,1.2) node[ddot] {}; 
 \draw (-1.2,0) node[dot] {} -- (0,1.2)node[ddot] {} -- (1.2,0) node[dia] {};}

\DeclareSymbol{31'}{-3}{
\draw (0,-1.2)node[dia] {} -- (0,0) node[ddot] {}-- (0.9, -1.2) node[dia] {}; 
\draw (0,-1.2)node[dia] {} -- (0,0) node[ddot] {}-- (-0.9, -1.2) node[dia] {}; 
\draw (0,0)node[ddot] {} -- (1.3,1.2) node[ddot] {}-- (1.3, 0) node[dia] {}; 
\draw  (1.3,1.2) node[ddot] {}-- (2.6, 0) node[dia] {}; 
}

\DeclareSymbol{32'}{-3}{
\draw (0,-1.2)node[dia] {} -- (0,0) node[ddot] {}-- (0.9, -1.2) node[dia] {}; 
\draw (0,-1.2)node[dia] {} -- (0,0) node[ddot] {}-- (-0.9, -1.2) node[dia] {}; 
\draw (0,0)node[ddot] {} -- (0,1.2) node[ddot] {}-- (0.9, 0) node[dia] {}; 
\draw (0,0)node[ddot] {} -- (0,1.2) node[ddot] {}-- (-0.9, 0) node[dot] {}; 
}

\DeclareSymbol{33'}{-3}{
\draw (2.6,-1.2)node[dia] {} -- (2.6,0) node[ddot] {}-- (3.5, -1.2) node[dia] {}; 
\draw (2.6,-1.2)node[dia] {} -- (2.6,0) node[ddot] {}-- (1.7, -1.2) node[dia] {}; 
\draw (0,0)node[dot] {} -- (1.3,1.2) node[ddot] {}-- (1.3, 0) node[dot] {}; 
\draw  (1.3,1.2) node[ddot] {}-- (2.6, 0) node[ddot] {}; 
}

\newtheorem*{ackno}{Acknowledgments}

\numberwithin{equation}{section}
\numberwithin{theorem}{section}

\begin{document}

\title[On norm inflation with general initial data
for the  cubic NLS]
{A remark on norm inflation 
with general initial data
for the cubic nonlinear Schr\"odinger equations
in negative Sobolev spaces}

\author
{Tadahiro Oh}

\address{
Tadahiro Oh, School of Mathematics\\
The University of Edinburgh\\
and The Maxwell Institute for the Mathematical Sciences\\
James Clerk Maxwell Building\\
The King's Buildings\\
Peter Guthrie Tait Road\\
Edinburgh\\ 
EH9 3FD\\
 United Kingdom}

\email{hiro.oh@ed.ac.uk}



\begin{abstract}
In this note, we 
consider  the ill-posedness issue for the cubic nonlinear Schr\"odinger equation.
In particular, we prove norm inflation
based at every initial condition
in negative Sobolev spaces
below or at the scaling critical regularity.

\end{abstract}

\maketitle

\section{Introduction}

We consider 
the  cubic  nonlinear Schr\"odinger equation (NLS): 
\begin{align}
\begin{cases}
i \dt u -    \Dl u \pm  |u|^{2}u  = 0\\
u|_{t = 0} = u_0, 
\end{cases}
\ (x, t) \in \M\times \R, 
\label{NLS0}
\end{align}

\noi
where $\M = \R^d$ or $\T^d$ with $\T = \R/\Z$.
The equation \eqref{NLS0} appears
in various physical settings: nonlinear optics, 
fluids, plasmas, 
and quantum field theory, 
and 
has been studied extensively
from both theoretical and applied points of view.
 See \cite{SULEM, Caz, TAO} for a general review on the subject.

Our main goal in this paper
is to study  the ill-posedness
issue of  \eqref{NLS0} in negative Sobolev spaces. 
We first recall the following scaling invariance for \eqref{NLS0};
 if $u(x, t)$ is a solution to \eqref{NLS0}
on $\R^d$, then
$u^\ld(x, t) := \ld^{-1} u (\ld^{-1}x, \ld^{-2}t)$
is also a solution to \eqref{NLS0} on $\R^d$ with scaled initial data.
Associated to this scaling invariance, 
we have the critical Sobolev regularity given by $s_\text{crit} := \frac d2 -1$.
While there is no scaling symmetry in the periodic setting, 
the heuristics provided by the scaling argument also plays an important role.
 It is commonly conjectured that an evolution equation is 
 well-posed  in $H^s$ for $s > s_\text{crit}$,
 while it is ill-posed for $s < s_\text{crit}$.
In fact, we have a good well-posedness theory 
\cite{GV, Tsutsumi, CW, BO1, HTT, Wang}
of \eqref{NLS0}
for $s \geq s_\text{crit}$, at least locally in time, 
with the exception of the $d = 1$ case.
See Remark \ref{REM:Wick1} below
for a brief discussion on the situation when $d = 1$.

On the other hand, 
\eqref{NLS0}
is known to be ill-posed on both $\R^d$ and $\T^d$ when $s < s_\text{crit}$.
 When $d = 1$, it is also ill-posed at the critical regularity 
 $s = s_\text{crit} = - \frac 12$.
More precisely, 
we have the following norm inflation;
given any $\eps > 0$, 
there exist a solution $u$ to \eqref{NLS0} on $\M$
and $t  \in (0, \eps) $ such that 
\begin{align}
 \| u(0)\|_{H^s(\M)} < \eps \qquad \text{ and } \qquad \| u(t)\|_{H^s(\M)} > \eps^{-1}.
\label{NI1}
 \end{align}

\noi
See \cite{CCT2b, CK, Oh, Kishimoto}.
This is a stronger notion
of ill-posedness than  the failure of continuity of the solution map at 0;
given any $\eps > 0$, 
there exist a solution $u$ to \eqref{NLS0}
and $t  \in (0, \eps) $ such that 
\[ \| u(0) - u_0\|_{H^s(\M)} < \eps \qquad \text{ and } \qquad \| u(t)\|_{H^s(\M)} \ges 1, \] 

\noi
where $u_0 = 0$.
Note  that   (the failure of)
continuity at $u_0$ 
can be discussed for  any $u_0 \in H^s(\M)$.
In a similar manner, we can discuss norm inflation based 
at any $u_0 \in H^s(\M)$.
This is precisely our goal in this paper.

\begin{theorem}\label{THM:1}
Given $d \in \N$, 
let  $\M = \R^d$ or $\T^d$.
Suppose that  $s \in \R$ satisfies
either 
\textup{(i)} $s \leq  - \frac 12$ 
when $d = 1$
or 
\textup{(ii)} $s <  0$ 
when $d \geq 2$.
Fix $u_0 \in H^s(\M)$.
Then, 
given any $\eps > 0$, 
there exist a %global 
solution $u_\eps$ to \eqref{NLS0} on $\M$
and $t  \in (0, \eps) $ such that 
\begin{align}
 \| u_\eps(0) - u_0 \|_{H^s(\M)} < \eps \qquad \text{ and } 
\qquad \| u_\eps(t)\|_{H^s(\M)} > \eps^{-1}.
\label{thm1}
\end{align}

\end{theorem}

When $u_0 = 0$, 
Theorem \ref{THM:1} is reduced to 
 the usual norm inflation
(based at the zero function) 
stated in \eqref{NI1}.
As a corollary to Theorem \ref{THM:1}, 
we obtain the following discontinuity of the solution map.

\begin{corollary}
Let $(s, d)$ be as in Theorem \ref{THM:1}.
Then, 
  for any $T>0$, 
the solution map $\Phi: u_0\in  H^s(\M) \mapsto u \in C([-T, T]; H^s(\M))$ 
to the cubic NLS \eqref{NLS0} is discontinuous
everywhere in $H^s(\M)$.
\end{corollary}

Norm inflation based at general initial data
was first studied in 
a recent paper by Xia \cite{Xia}
in the context of the
nonlinear wave equations on $\T^3$.
The argument in \cite{Xia}
is based on the (dispersionless)
ODE approach in the spirit of 
Christ-Colliander-Tao
\cite{CCT2b} and Burq-Tzvetkov \cite{BT1}.

We prove
Theorem \ref{THM:1}
by
 a Fourier analytic approach. 
Recently,  Iwabuchi-Ogawa \cite{IO} 
developed a method  for proving ill-posedness of 
evolution equations, exploiting high-to-low energy transfer
in the first Picard iterate.
This method is built upon the previous work by Bejenaru-Tao \cite{BT}
and is developed further to cope with a  wider class of equations, 
utilizing  modulation spaces.
See Kishimoto \cite{Kishimoto}
for a norm inflation result on  \eqref{NLS0} via this method.
In the following, we implement a refinement of 
the argument in \cite{IO, Kishimoto}
and prove Theorem \ref{THM:1}.

In \cite{IO, Kishimoto}, 
the (scaled) modulation space $M_{2, 1}$
and its algebra property played an important role.
In the proof of Theorem \ref{THM:1}, 
we simply use the Wiener algebra $\F L^1(\M)$ as a replacement.
Given $\M = \R^d$ or $\T^d$, 
let   $\ft \M$ denote the Pontryagin dual of $\M$, 
i.e.~
\begin{align}
\ft \M = \begin{cases}
 \R^d & \text{if }\M = \R^d, \\
 \Z^d & \text{if } \M = \T^d.
\end{cases}
\label{dual}
\end{align}

\noi
When $\ft \M = \Z^d$, 
we endow it with the counting measure.
We then define
the Fourier-Lebesgue space $\F L^{p}(\M)$  by the norm:
\begin{align*}
\|f \|_{\F L^{p}(\M)} = \big\|  \ft f \big\|_{L^{p}(\M)}.
\end{align*}

\noi
In particular, 
$\F L^1(\M)$ corresponds to 
the Wiener algebra.
By the algebra property, it is easy to see that 
 \eqref{NLS0} is analytically locally well-posed
 in $\F L^1(\M)$.

Another key ingredient in \cite{IO, Kishimoto}
is the power series expansion
of a solution $u$ to \eqref{NLS0} with $u|_{t = 0} = u_0$:
\begin{align*}
 u & 
 = \sum_{j = 0}^\infty \Xi_j (u_0),
 \end{align*}

\noi
where $\Xi_j(u_0)$ denotes 
homogeneous multilinear terms (in $u_0$) of 
degree $2j+1$.
In \cite{IO, Kishimoto}, 
 $\Xi_j$ was defined by a recursive relation
of the form: 
\begin{align}
 \Xi_j(u_0) = \sum_{\substack{j_1 + j_2 + j_3 = j\\j_1, j_2, j_3 \geq 1}}
\I[ \Xi_{j_1}(u_0),  \Xi_{j_2}(u_0), \Xi_{j_3}(u_0)],
\label{rec1}
\end{align}

\noi
where $\I$ is the trilinear Duhamel integral operator defined in \eqref{Duhamel1}.
Then, nonlinear estimates were proved inductively.
In the following, we instead define $\Xi_j$ directly
via 
the power series expansion indexed by trees.
This allows us to establish nonlinear estimates without an induction.
See Section \ref{SEC:2}.
Such a power series expansion indexed by trees
 is more suitable in handling an expression such as
$\Xi_j(u_0 + \phi) - \Xi_j(\phi)$,
which is the main new ingredient
for showing norm inflation based at a general initial condition $u_0$.

When $ d \geq 3$, 
we have $s_\text{crit} > 0$
and thus 
Theorem \ref{THM:1} leaves a gap $(0, s_\text{crit})$.\footnote{In view of the mass conservation, 
\eqref{thm1} can not hold for $s = 0$.
Moreover, when $s_\text{crit}>  0$, 
we do not have norm inflation at $s = s_\text{crit}$
in view of the well-posedness in the scaling critical spaces.}
In this case, one needs to exploit low-to-high energy transfer
to prove an analogue of Theorem \ref{THM:1} for $s \in (0, s_\text{crit})$.
By employing the ODE approach as in \cite{CCT2b, Oh}, 
we plan to address this issue
in a forthcoming work.

\medskip

We conclude this introduction by several remarks.

\begin{remark}\rm
We point out that Theorem \ref{THM:1}
also holds on a general irrational torus $\T_\al^d = \prod_{j = 1}^d (\R/\al_j \Z)$, 
$\al_j > 0$.
This is due to the fact that
the proof of Theorem \ref{THM:1} does not use any fine arithmetic property
of frequency interactions.
Indeed, we only use the following trivial estimate:
\begin{align*}
\big||\xi|^2 - |\xi_1|^2 + |\xi_2|^2 - |\xi_3|^2\big| \les N^2
\end{align*}

\noi
for $\xi = \xi_1 - \xi_2 + \xi_3$ with $|\xi_j| \les N$, $j = 1, 2, 3$.
See \eqref{Hs33} below.

\end{remark}

\begin{remark}\rm
In Theorem \ref{THM:1}, 
we only considered the cubic nonlinearity
for simplicity of the presentation.
Our method is elementary
and can be applied to other power-type
nonlinearities.

We can also consider 
norm inflation based at general initial data for 
the cubic fractional NLS:
\begin{align*}
i \dt u +  (- \Dl^2)^\al u \pm  |u|^{2}u  = 0. 
\end{align*}

\noi
In \cite{CP}, Choffrut-Pocovnicu recently applied
the argument in \cite{IO, Kishimoto}
and proved norm inflation \eqref{NI1} (at the zero function)
for the cubic half-wave equation (i.e. $\al = \frac 12$)
on $\R$ and $\T$ when $s < 0$.
By adapting the proof of Theorem \ref{THM:1}, 
we can extend this result in \cite{CP}
to  norm inflation 
based at general initial data
as in Theorem \ref{THM:1}.\footnote{After the completion of this manuscript, the author has learned that Choffrut and Pocovnicu \cite{CP} extended their norm inflation result to the general cubic fractional NLS for any $\al > 0$. In particular when $\al > 1$, they established norm inflation above the scaling critical regularity.}
\end{remark}

\begin{remark}\label{REM:Wick1}\rm
The equation \eqref{NLS0}
is also invariant under the Galilean symmetry.
This symmetry preserves the $L^2$-norm of solutions,
giving rise to another critical regularity  $s_\text{crit}^\infty := 0$.
In particular, 
this critical regularity plays an important role when $d = 1$.
While
 \eqref{NLS0} is globally well-posed
in $L^2(\R)$ \cite{Tsutsumi},  
 it is known to be `mildly ill-posed' below $L^2(\R)$
in the sense of the failure 
of local uniform continuity in negative Sobolev spaces \cite{KPV, CCT1}.

On $\T$, the contrast between $s \geq 0$ and $s< 0$ is more drastic.
On the one hand, 
\eqref{NLS0} is  globally well-posed in $L^2(\T)$ \cite{BO1}.
On the other hand,  it is ill-posed in negative Sobolev spaces.
Christ-Colliander-Tao \cite{CCT2} and Molinet \cite{MOLI}
showed discontinuity of the solution map in negative Sobolev spaces.
Moreover,  Guo-Oh \cite{GO} proved non-existence of solutions for \eqref{NLS0}
with initial data lying strictly in a negative Sobolev space.

As an alternative model to \eqref{NLS0}
in low regularity setting on $\T^d$, we have 
the following  Wick ordered cubic NLS:
\begin{equation}
\label{NLS1} 
\textstyle
i \dt u - \Dl u \pm \big( |u|^2 -2  \int_{\T^d}  |u|^2 dx\big) u = 0.
\end{equation}

\noi
This equation  is gauge equivalent to \eqref{NLS0} in $L^2(\T^d)$,
while it behaves better than \eqref{NLS0} outside $L^2(\T^d)$.
See \cite{OS} for more discussion on this issue.
By slightly modifying the argument, 
Theorem \ref{THM:1} also holds for \eqref{NLS1} on $\T^d$.
See Remark \ref{REM:Wick2}.

Lastly, we point out that 
 the well-posedness issue of \eqref{NLS0} on $\R$
 and of \eqref{NLS1} on $\T$ 
 in $H^s$ for $-\frac 12 < s < 0$
is widely open.
See \cite{KT1, CCT3, KT2, GO} for partial results on this problem.

\end{remark}

\begin{remark}\label{REM:prob}
\rm

Given initial data below a scaling critical regularity $s_\text{crit}$, 
it is often possible to 
suitably randomize initial data
and 
construct solutions
in a probabilistic manner.
See \cite{BO96, BT1, BOP1, BOP2}.
Such probabilistic well-posedness results
do not 
yield (deterministic) continuous dependence in $H^s$ with $s < s_\text{crit}$.\footnote{See also \cite{BT3} for the notion of probabilistic continuous dependence.}
In particular, 
norm inflation based at general data 
such as 
Theorem \ref{THM:1}
does not contradict 
these probabilistic results.

It is worthwhile to note that these probabilistic
constructions of solutions yield  mild continuous dependence
in a smoother regularity.
Given $u_0 \in H^s$ with $s < s_\text{crit}$, 
the argument in \cite{BO96, BT1, BOP1, BOP2} allows
us to randomize $u_0$
and construct 
a set $\Si_{u_0}$ of full probability 
such that, given any $\o \in \Si_{u_0}$, 
there exists 
a random solution 
$u^\o$ of the form 
$u^\o = S(t) u_0^\o + v^\o$, 
where $S(t) u_0^\o$ is the linear solution 
with the randomized initial data $u_0^\o$
and $v^\o$ denotes the random nonlinear part
belonging to a smoother space; 
$v^\o(t) \in H^\s$ for some $\s \geq s_\text{crit}$.
Moreover, 
there exists $\eps _0 > 0$ such that 
if a deterministic function $w_0$
satisfies 
\begin{align}
\| u_0^\o - w_0\|_{H^\s} < \eps_0
\label{conti0}
\end{align}

\noi
for some $\o \in \Si_{u_0}$, 
then 
there exists a solution $w$  with $w|_{t = 0} = w_0$, satisfying
\begin{align}
\| u^\o(t)  - w(t)\|_{H^\s}\les \| u_0^\o - w_0\|_{H^\s}, 
\label{conti}
\end{align}

\noi
for $t \in [-T_\o, T_\o]$.
Namely, we have some kind of continuous dependence
in this probabilistic setting, 
but in a smoother regularity.
Note that this mild continuous dependence \eqref{conti} does not contradict Theorem \ref{THM:1}
since they 
are in different regularity regimes:
 $s < s_\text{crit} \leq \s$. 
In particular,  \eqref{conti0} is much more restrictive than the first estimate in \eqref{thm1}.

\end{remark}

\medskip
The defocusing/focusing nature of 
the equation does not play any role.
Hence, we assume that it is defocusing
(with the $+$\,sign in \eqref{NLS0}) in the following.
Moreover, in view of the time reversibility of the equation, 
we only consider positive times.

\section{Preliminary analysis}
\label{SEC:2}

In this section, we first discuss the local well-posedness 
of \eqref{NLS0} in the Wiener algebra $\F L^1$.
Then, we express solutions in  power series in terms of initial data, 
where the summation ranges over all finite ternary trees.
We then establish basic nonlinear estimates
on the multilinear terms arising in the power series expansion.

\subsection{Power series expansion indexed by trees}

We define the Duhamel integral  operator
$\I$ by 
\begin{align}
\I[u_1, u_2, u_3](t)
:= i 
\int_0^t S(t - t') [u_1 (t')\cj{u_2(t')} u_3(t')]dt',
\label{Duhamel1}
\end{align}

\noi
where $S(t) = e^{-it \Dl}$ is the linear Schr\"odinger propagator.
When all the three arguments $u_1, u_2$, and $u_3$ are identical, 
we use the following shorthand notation:
\begin{align}
\I^3[u] := \I[u, u, u].
\label{Duhamel2}
\end{align}

\noi
We say that $u$ is a solution to \eqref{NLS0}
with $u|_{t = 0} = u_0$
if $u$ satisfies the following Duhamel formulation:
\begin{align}
u(t) = S(t) u_0 + \I^3[u](t).
\label{Duhamel3}
\end{align}

\noi
We first state the local well-posedness of \eqref{NLS0} 
in the Wiener algebra $\F L^1$.

\begin{lemma}\label{LEM:LWP}

The cubic NLS \eqref{NLS0} is locally well-posed
in the Wiener algebra $\F L^1$.
More precisely, given $u_0 \in \F L^1$, 
there exist $T \sim \| u_0\|_{\F L^1}^{-2}>0 $ and a unique solution $ u \in C([-T, T]; \F L^1)$
satisfying \eqref{Duhamel3}.

\end{lemma}

In view of the unitarity of $S(t)$ in $\F L^1$ and 
the algebra property of $\F L^1$, 
Lemma \ref{LEM:LWP} follows
from a standard fixed point argument.
We omit details.

Let $\phi \in \F L^1$.
Then, (the proof of) Lemma \ref{LEM:LWP}   guarantees the convergence of the following Picard iteration
scheme: 
\begin{align}
P_0(\phi) = S(t) \phi
\qquad \text{and}\qquad
P_{j}(\phi) = S(t) \phi + \I^3[P_{j-1}(\phi)], \ \  j \in \N, 
\label{power2}
\end{align}

\noi
at least for short times. 
It follows from  \eqref{Duhamel2} and \eqref{power2}  that $P_j$ consists of multilinear terms
of degrees at most $3^j$ (in $\phi$).
In the following, we discuss a more general recursive scheme
and express a solution in a power
series indexed by trees.
Christ \cite{CH2} implemented such a power series expansion
of solutions to the Wick ordered cubic NLS \eqref{NLS1}
on $\T$ in a low regularity setting.
Since we work only with smooth functions, 
our presentation is much simpler.

We introduce the following notion of (ternary) trees.
As in  \cite{CH2},
our trees 
refer to a particular subclass of usual trees with the following properties:

\begin{definition} \label{DEF:tree} \rm
(i) Given a partially ordered set $\TT$ with partial order $\leq$, 
we say that $b \in \TT$ 
with $b \leq a$ and $b \ne a$
is a child of $a \in \TT$,
if  $b\leq c \leq a$ implies
either $c = a$ or $c = b$.
If the latter condition holds, we also say that $a$ is the parent of $b$.

\smallskip 

\noi
(ii) 
A tree $\TT$ is a finite partially ordered set,  satisfying
the following properties:
\begin{itemize}
\item Let $a_1, a_2, a_3, a_4 \in \TT$.
If $a_4 \leq a_2 \leq a_1$ and  
$a_4 \leq a_3 \leq a_1$, then we have $a_2\leq a_3$ or $a_3 \leq a_2$,

\item
A node $a\in \TT$ is called terminal, if it has no child.
A non-terminal node $a\in \TT$ is a node 
with  exactly three children,

\item There exists a maximal element $r \in \TT$ (called the root node) such that $a \leq r$ for all $a \in \TT$,

\item $\TT$ consists of the disjoint union of $\TT^0$ and $\TT^\infty$,
where $\TT^0$ and $\TT^\infty$
denote  the collections of non-terminal nodes and terminal nodes, respectively.
\end{itemize}

\end{definition}

\noi
Note that the number $|\TT|$ of nodes in a tree $\TT$ is $3j+1$ for some $j \in \mathbb{N}\cup\{0\}$,
where $|\TT^0| = j$ and $|\TT^\infty| = 2j + 1$.
Let us denote  the collection of trees in the $j$th generation (i.e.~with $j$ parental nodes) by $\BT(j)$, i.e.
\begin{equation*}
\BT(j) := \big\{ \TT : \TT \text{ is a tree with } |\TT| = 3j+1 \big\}.
\end{equation*}

\noi
Then, we have the following exponential bound
on the number 
$\# \BT(j)$ of 
 trees in the $j$th generation.

\begin{lemma}\label{LEM:tree}
Let $\BT(j)$ be as above.
Then, there exists $C_0 >0$ such that 
\begin{align}
\# \BT(j) \leq C_0^j
\label{tree0a}
\end{align}

\noi
for all $j \in \mathbb{N}\cup\{0\}$.
\end{lemma}

This lemma is a basic fact about ternary trees.
We present the proof for the sake of completeness.

\begin{proof}
Clearly, we have $\#\BT(0) = \#\BT(1) = 1$.
We now consider   $j \geq 2$.
From Definition \ref{DEF:tree}, 
we have  the following identity:
\begin{align}
\# \BT(j)  = \sum_{
\substack{j_1 + j_2 + j_3 = j - 1\\j_1, j_2, j_3 \geq 0}}
\# \BT(j_1)\cdot \# \BT(j_2)\cdot \# \BT(j_3).
\label{tree0b}
\end{align}

\noi
Assume the following stronger estimate:
\begin{align}
\# \BT(k) \leq \frac{C_0^k}{(1+k)^2}
\label{tree0c}
\end{align}

\noi
for all $k \leq j-1$.
Then, 
 from \eqref{tree0b}
and \eqref{tree0c}
with $3 \max(j_1+1, j_2+1, j_3+1) \ge j+1$, 
we have
\begin{align}
\# \BT(j)  
& \leq \sum_{\substack{j_1 + j_2 + j_3 = j - 1\\j_1, j_2, j_3 \geq 0}}
\frac{C_0^{j_1}}{(1+j_1)^2}\frac{C_0^{j_2}}{(1+j_2)^2}\frac{C_0^{j_3}}{(1+j_3)^2} \notag\\
& \leq 3^2 \bigg(\sum_{k  \geq 0}
\frac{1}{(1+k)^2}\bigg)^2
\cdot \frac{C_0^{j-1}}{(1+j)^2}
\leq 
\frac{C_0^{j}}{(1+j)^2}, 
\end{align}

\noi
where the last inequality holds 
by choosing $C_0 =   3^2 \big(\sum_{k  \geq 0}
\frac{1}{(1+k)^2}\big)^2< \infty$.
This proves \eqref{tree0c} for $k = j$.
By induction, \eqref{tree0c}  holds for all $j \in \N \cup\{0\}$,
yielding \eqref{tree0a}
for all $j \in \N \cup\{0\}$.
\end{proof}

\begin{remark}\label{REM:Kishimoto}
\rm

In \cite{Kishimoto}, 
a similar counting argument is needed
to control the number of terms appearing
in the recursive definition \eqref{rec1}
of $\Xi_j$.

\end{remark}

\medskip

Next, we express the solution $u$ constructed in Lemma \ref{LEM:LWP}
in a power series indexed by trees.
Fix $\phi \in \F L^1$.
Given a tree $\TT \in \BT(j)$,
$j \in \N \cup\{0\}$, 
we associate a multilinear\footnote{By a multilinear operator, 
we mean an operator which is linear or conjugate linear with respect to each argument,
i.e.~linear over real numbers.} 
operator (in $\phi$) by the following rules:
\begin{itemize}
\item Replace a non-terminal node ``\,$\<1'>$\,'' 
by the Duhamel integral operator $\I$ defined in \eqref{Duhamel1}
with its three children as arguments $u_1, u_2$, and $u_3$, 

\item Replace a terminal node ``\,$\<1>$\,'' 
by the linear solution $S(t) \phi$. 
\end{itemize}

\noi
In the following, we denote this mapping
from $ \bigcup_{j = 0}^\infty \BT(j)$ to  $\mathcal{D}'(\M\times [-T, T])$
by $\Psi_\phi$.

For example, 
 $\Psi_\phi$ maps
the trivial  tree ``\,$\<1>$\,'',  consisting only of the root node
to the linear solution $S(t) \phi$.
Namely, we have $\Psi_\phi(\,\<1>\,) = S(t) \phi$.
Similarly, we have 
\begin{align}
\Psi_\phi( \<3>) & =
\I^3[S(t) \phi]
\label{tree1a}
\\
\Psi_\phi\big( \<31>\big) & =
\I[\I^3[S(t) \phi], S(t) \phi, S(t) \phi], 
\notag
\end{align}

\noi
where $\I^3$ is as in \eqref{Duhamel2}.
In view of the algebra property of $\F L^1$ 
along with the continuity and unitarity of $S(t)$, 
we have
$\Psi_\phi(\TT) \in C([-T, T]; \F L^1)$
for any tree $\TT$,  provided $\phi \in \F L^1$.
Note that, if  $\TT \in \BT(j)$, 
then $\Psi_\phi(\TT) $ is  $(2j+1)$-linear in $\phi$.
Lastly, we define $\Xi_j$ by 
\begin{align}
\Xi_j (\phi)
: =  \sum_{\TT \in \BT(j)} \Psi_\phi (\TT).
\label{tree1}
 \end{align}

Now, let us extend the definition of $\Psi_\phi$
by adding the following rules
for another kind of terminal node ``$\,\<1''>\,$'':
\begin{itemize}
\item Replace a star-shaped terminal node ``$\,\<1''>\,$''
by the solution $u$, 

\item Extend the definition of $\Psi_\phi$
to formal sums of trees by imposing linearity.
\end{itemize}

\noi
Then, the Duhamel formulation \eqref{Duhamel3} can be represented by 
\begin{align}
 \<1''>\, = \, \<1>\, + \, \<3'>.
\label{R1}
 \end{align}

\noi
By recursively applying \eqref{R1}
and eliminating the occurrence of ``$\,\<1''>\,$''
from younger trees, we have 
\begin{align}
 \<1''>\, 
 & = \, \<1>\, + \, \<3''>\, + \,\<31'>\, \notag\\
 & = \, \<1>\, + \, \<3'''> \,+\, \<31'> \,+\,   \<32'> \notag\\
 & = \, \<1>\, + \, \<3> \,+\,\<31'>\,+ \,\<32'>\,+\,   \<33'> 
 \notag\\
 &  = \cdots = \, \<1>\, + \, \<3> \,+\,\<31>\,+ \,\<32>\,+\,   \<33> 
 + \cdots.
\label{R2}
 \end{align}

\noi
This iterative scheme leads, under $\Psi_\phi$, 
to  the following (formal) power series
expansion of the solution $u$ to \eqref{NLS0} with $u|_{t = 0} = \phi \in \F L^1$:
\begin{align*}
 u & 
 = \sum_{j = 0}^\infty \Xi_j (\phi)
 = \sum_{j = 0}^\infty \sum_{\TT \in \BT(j)} \Psi_\phi (\TT)\notag\\
& = 
\Psi_\phi(\, \<1>\,) 
+ \Psi_\phi( \<3>) +
\Psi_\phi\big( \<31>\big) 
+
\Psi_\phi\big( \<32>\big) 
+ \Psi_\phi\big( \<33>\big) 
+\cdots.
 \end{align*}

\noi
It is easy to see that 
 the above power series converges (absolutely)
in $C([-T, T]; \F L^1)$
and the first equality holds
as long as $T = T(\|\phi\|_{\F L^1}) \sim \| \phi\|_{\F L^1}^{-2} > 0$ is sufficiently small.
This follows from  Lemma \ref{LEM:nonlin1} below
along with Lemma \ref{LEM:tree}.
See also \cite[Proposition 6.4.2]{CH2}.

Given a tree $\TT \in \BT(j)$, 
label its terminal nodes
by $a_1, \dots, a_{2j+1}$
(say, by moving from left to right in the planar graphical representation of the tree).
Given functions $\phi_1, \dots, \phi_{2j+1} \in \F L^1$,
we also define $\Psi (\TT; \phi_1, \dots, \phi_{2j+1})$
by the following rules:
\begin{itemize}
\item Replace a non-terminal node $a \in \TT^0$
by the Duhamel integral  operator $\I$
with its three children as arguments $u_1, u_2$, and $u_3$, 

\item Replace  terminal nodes $a_k \in \TT^\infty$
by the linear solutions $S(t) \phi_k$, $k = 1, \dots, 2j+1$. 
\end{itemize}

\noi
In particular, we have
$\Psi_\phi(\TT) = \Psi (\TT; \phi, \dots, \phi)$.

\subsection{Basic multilinear estimates}

In this subsection, we state basic multilinear estimates
on $\Xi_j$.
These will be used in establishing $H^s$-estimates on $\Xi_j(u_0 + \phi)$
in Section \ref{SEC:proof}.

\begin{lemma}\label{LEM:nonlin1}
There exists $C >0$ such that 
\begin{align}
\| \Xi_j (\phi)(t) \|_{\F L^1}
& \leq C^j t^{j} 
\| \phi\|_{\F L^1}^{2j+1},  
\label{nonlin11a}\\
\| \Xi_j (\phi)(t) \|_{\F L^\infty}
& \leq C^j t^{j} 
\| \phi\|_{\F L^1}^{2j-1} 
\| \phi\|_{L^2}^{2} 
\label{nonlin11}
\end{align}
	
\noi
for 
all $\phi \in \F L^{1}(\M)$
and 
all $j \in \N$.

\end{lemma}

\begin{proof}

Given $\TT \in \BT(j)$, 
$\Psi_\phi(\TT)$ contains $j = |\TT^0|$ 
(partially) iterated time integrations 
over a subset of $[0, t]^j$.
Then, 
from  the unitarity of $S(t)$ in $\F L^1$ and $L^2$
and Young's inequality, we have 
\begin{align}
\| \Psi_\phi (\TT)(t)\|_{\F L^1}
& \les t^{j} 
\| \phi\|_{\F L^1}^{2j+1}  
\label{nonlin12a},\\
\| \Psi_\phi (\TT)(t)\|_{\F L^\infty}
& \les t^{j} 
\| \phi\|_{\F L^1}^{2j-1} 
\| \phi\|_{L^2}^{2} 
\label{nonlin12}.
\end{align}

\noi
Hence, the desired estimates
\eqref{nonlin11a} and 
\eqref{nonlin11}
follow from 
\eqref{nonlin12a} and
\eqref{nonlin12}
with 
 \eqref{tree1} and 
Lemma \ref{LEM:tree}.
\end{proof}

\begin{lemma}\label{LEM:nonlin2}
There exists $C >0$ such that 
\begin{align}
\| \Xi_j (u_0 + \phi)(t)
- \Xi_j ( \phi)(t)\|_{\F L^p}
\leq C^j t^{j} 
\|u_0 \|_{\F L^p}
\big(
\| u_0\|_{\F L^1}^{2j}
+ \| \phi\|_{\F L^1}^{2j}\big)
\label{nonlin21}
\end{align}
	
\noi
for 
all $u_0 \in \F L^p(\M)\cap \F L^1(\M) $
with $1\leq p \leq \infty$, 
$ \phi \in \F L^{1}(\M)$, 
and  $j \in \N$.

\end{lemma}

\begin{proof}
From  \eqref{tree1}
and the multilinearity of $\Psi_\phi(\TT)$ in $\phi$, we have
\begin{align}
 \Xi_j (t; u_0 + \phi) - \Xi_j (\phi)
& =  \sum_{\TT \in \BT(j)} \big( \Psi_{u_0+ \phi} (\TT) - \Psi_\phi (\TT)\big)  \notag \\
& =  \sum_{\TT \in \BT(j)} 
\sum_{\phi_1, \dots, \phi_{2j+1}}
 \Psi (\TT; \phi_1, \dots, \phi_{2j+1}).
\label{nonlin22}
\end{align}

\noi
Here, 
the second summation in $\phi_1, \dots, \phi_{2j+1}$
takes over all possible combinations 
of $\phi_k = u_0$ or $\phi_k = \phi$
with at least one occurrence of $u_0$.
Note that we have
\begin{align}
\bigg|\sum_{\phi_1, \dots, \phi_{2j+1}} 1\, \bigg|
\leq 2^{2j+1}.
\label{nonlin23}
\end{align}

As in the proof of Lemma \ref{LEM:nonlin1},  given $\TT \in \BT(j)$, 
 there are $j = |\TT^0|$ 
time integrations 
in $ \Psi (\TT; \phi_1, \dots, \phi_{2j+1})$
over a subset of $[0, t]^j$.
Then, 
from  the unitarity of $S(t)$ in $\F L^p$ and $\F L^1$
and Young's inequality, we have 
\begin{align}
\| \Psi (\TT; \phi_1, \dots, \phi_{2j+1})(t)\|_{\F L^p}
\les t^{j}
 \| \phi_{k_*}\|_{\F L^p}
 \prod_{\substack{k = 1\\k \ne k_*}}^{2j+1} \| \phi_k\|_{\F L^1}.
\label{nonlin24}
\end{align}

\noi
Then, 
the desired estimate
\eqref{nonlin21}
follows from 
\eqref{nonlin22}, 
\eqref{nonlin23}, 
\eqref{nonlin24},
and 
Lemma \ref{LEM:tree}.
\end{proof}

\section{Proof of Theorem \ref{THM:1}}
\label{SEC:proof}

In this section, we present the proof of Theorem \ref{THM:1}.
We first reduce Theorem \ref{THM:1} to the following proposition.
Given 
 $\M = \R^d$ or $\T^d$, 
 we use $\S(\M)$ to denote the class of Schwartz functions if $\M = \R^d$
 and the class of $C^\infty$-functions if $\M = \T^d$.

\begin{proposition}\label{PROP:main}
Let $(s, d)$ be as in Theorem \ref{THM:1}.
Fix $u_0 \in \S(\M)$.
Then, 
given any $n \in \N$, 
there exist a %global 
solution $u_n$ to \eqref{NLS0} 
and $t_n  \in \big(0, \frac 1n\big) $ such that 
\begin{align}
 \| u_n(0) - u_0 \|_{H^s(\M)} < \tfrac 1n \qquad \text{ and } 
\qquad \| u_n(t_n)\|_{H^s(\M)} > n.
\label{main1}
\end{align}

\end{proposition}

We first prove Theorem \ref{THM:1}, assuming Proposition \ref{PROP:main}.
Given $u_0 \in H^s$, 
let $\{u_{0, k}\}_{k\in \N} \subset \S(\M)$
be a sequence of smooth functions converging to $u_0$ in $H^s$, 
satisfying 
\begin{align}
\| u_{0, k} - u_{0} \|_{H^s(\M)} < \tfrac 1k.
\label{thm1a}
\end{align}

\noi
It follows from Proposition \ref{PROP:main}
that, for each $k \in \N$, 
there exists a sequence 
$\{ u_{k, n}\}_{n \in \N}$ of %global 
solutions to \eqref{NLS0}
such that 
\begin{align}
 \| u_{k, n}(0) - u_{0, k} \|_{H^s(\M)} < \tfrac 1n \qquad \text{ and } 
\qquad \| u_{k, n}(t_n)\|_{H^s(\M)} > n
\label{thm1b}
\end{align}

\noi
for all $n \in \N$.
Now, given $\eps >0$, set $u_\eps = u_{n, n}$
for some $n \geq 2 \eps^{-1}$.
Then, \eqref{thm1} follows from \eqref{thm1a} and \eqref{thm1b}.
This proves Theorem \ref{THM:1}.

The remaining part of this paper 
is devoted to the proof of Proposition \ref{PROP:main}.
In the following, we fix $u_0 \in \S(\M)$
and may suppress 
the dependence of various constants on $u_0$.

\subsection{Multilinear estimates}
In this subsection, 
we establish elementary, yet important multilinear estimates
in the relevant Sobolev $H^s$-norm.
In particular, 
Proposition \ref{PROP:Hs3}
provides a lower bound
on the trilinear term $\Xi_1(\phi_n)$
for appropriately chosen functions $\phi_n$.

Given $n \in \N$, 
fix $N = N(n) \gg 1$ (to be chosen later).
We define $\phi_n$ by setting
\begin{align}
\ft \phi_n (\xi)= R\big\{\ind_{Ne_1+ Q_A} (\xi)+ \ind_{2Ne_1+ Q_A}(\xi)\big\}, 
\label{phi1}
\end{align}
	
\noi
where $Q_A = \big[-\frac A2, \frac A2\big)^d$, 
$e_1 = (1, 0, \dots, 0)$, 
$R = R(N) \geq 1 $,  and $A = A(N)\gg 1$, satisfying
\begin{align}
RA^d  \gg \|u_0\|_{\F L^1}, 
\qquad \text{and}
\qquad A\ll N, 
\label{phi1a}
\end{align}

\noi
 are to be chosen later.
 Note that we have
\begin{align}
\| \phi_n\|_{H^s} \sim R A^\frac{d}{2} N^s
\qquad \text{and} \qquad \| \phi_n\|_{\F L^1} \sim R A^d, 
\label{phi2}
\end{align}

\noi
for any $s \in \R$.
Lastly, 
set 
\begin{align}
u_{0, n} = u_0 + \phi_n.
\label{phi3}
\end{align}

\noi
Let $u_n$ be 
 the corresponding solution to \eqref{NLS0} with $u_n|_{t = 0} = u_{0,n}$.
Lemmas \ref{LEM:tree} and \ref{LEM:nonlin1} with \eqref{phi2} guarantee the convergence
of the following power series expansion:
\begin{align}
 u_n & 
 = \sum_{j = 0}^\infty \Xi_j (u_{0, n}) 
  = \sum_{j = 0}^\infty \Xi_j (u_{0} + \phi_n), 
\label{tree3}
 \end{align}

\noi
on $[-T, T]$, 
as long as 
\[
T \les (\| u_0\|_{\F L^1} + R A^d)^{-2}
\sim (R A^d)^{-2},\]

\noi
where the last equivalence follows from \eqref{phi1a}.
Our main goal is to show that $u_n$ 
satisfies \eqref{main1}
by estimating each of $\Xi_j (u_{0, n})$ in \eqref{tree3}.

We now state the nonlinear estimates.
Keep in mind that implicit constants in  Lemmas \ref{LEM:Hs1}
and \ref{LEM:Hs2}
depend on (various norms of) $u_0$.

\begin{lemma}\label{LEM:Hs1}

Let
$\phi_n$ and  $u_{0, n}$ be as in \eqref{phi1} and \eqref{phi3}.
Let $s < 0$.
Then, the following estimates hold: 
\begin{align}
\| u_{0, n} - u_0\|_{H^s} & \les R A^{\frac {d}{2}}N^s, 
\label{Hs11}\\
\| \Xi_0(u_{0,n})\|_{H^s} & \les 1 + R A^{\frac {d}{2}}N^s,  \label{Hs12}\\
\| \Xi_1 (u_{0, n})(t) - \Xi_1 ( \phi_n)(t)\|_{H^s}
& \les  t \| u_0\|_{L^2} R^2A^{2d}.
\hphantom{XXX}
\label{Hs13}
\end{align}

\end{lemma}

\begin{proof}
Noting that $\phi_n = u_{0, n} - u_0$, 
the first two estimates \eqref{Hs11} and \eqref{Hs12}
follow from  \eqref{phi2}
and the unitarity of $S(t)$.
Since $s < 0$, 
it follows from 
Lemma \ref{LEM:nonlin2} (with $p = 2$)
with  \eqref{phi2} and \eqref{phi1a} that 
\begin{align*}
\| \Xi_1 (u_{0, n})(t) - \Xi_1 ( \phi_n)(t)\|_{H^s}
& \les  t \| u_0\|_{L^2}\big(\|u_0\|_{\F L^1}^2 + \| \phi_n\|_{\F L^1}^2\big)\notag\\
& \les t \| u_0\|_{L^2}(1 + R^2A^{2d})\notag\\
& \les t \| u_0\|_{L^2} R^2A^{2d}.
\end{align*}

\noi
This proves \eqref{Hs13}.
\end{proof}

\begin{lemma}\label{LEM:Hs2}

Let
$\phi_n$ and  $u_{0, n}$ be as in \eqref{phi1} and \eqref{phi3}.
Let $s < 0$.
Then, there exists $C >0$ such that 
\begin{align}
\| \Xi_j (u_{0, n})(t)\|_{H^s}
\leq C^j t^{j} 
(RA^d)^{2j} \big\{ R f(A)+ \| u_0\|_{L^2}\big\}
\label{Hs21}
\end{align}
	
\noi
for any $j \in \N$, 
where $f(A)$ is given by 
\begin{align}
f(A) = 
\begin{cases}
1, & \text{if } s < -\frac d 2, \\
(\log A)^\frac 12, & \text{if } s = -\frac d 2, \\
A^{\frac{d}{2}+s} & \text{if } s > -\frac d 2.
\end{cases}
\label{Hs21a}
\end{align}

\end{lemma}
	
\begin{proof}

From \eqref{phi1}, 
we see that $\supp \ft \phi_n$
consists of two disjoint cubes of volume $\sim A^d$.
Given $\TT \in \BT(j)$, 
$\Psi_\phi(\TT)$ is basically 
a $(2j+1)$-fold product of 
$S(t) \phi$ and its complex conjugate under
some integral operator in time.
Hence, the
spatial support of $\F[ \Psi_\phi(\TT)]$
consists of (at most) $2^{2j+1}$
cubes of volume $\sim A^d$.
Then, 
from \eqref{tree1} and Lemma \ref{LEM:tree}, we have
\begin{align*}
\big|\supp \F[\Xi_j(\phi_n)]\big| \leq C^j A^d \leq c | C^j Q_A|
\end{align*}

\noi
for some $c, C>0$.
Since $s < 0$, 
$\jb{\xi}^s$ is a decreasing function in  $|\xi|$.
Hence, we obtain
\begin{align}
\| \jb{\xi}^s\|_{L^2_\xi(\supp \F[\Xi_j (\phi_n)])}
& \leq \| \jb{\xi}^s\|_{L^2_\xi(c C^j Q_A)}\notag\\
& \les 
\begin{cases}
1, & \text{if } s < -\frac d 2, \\
C^j (\log A)^\frac 12, & \text{if } s = -\frac d 2, \\
C^j A^{\frac{d}{2}+s} & \text{if } s > -\frac d 2.
\end{cases}
\label{Hs22}
\end{align}

\noi
By Lemma \ref{LEM:nonlin1} with \eqref{phi2}
and \eqref{Hs22}, 
we have 
\begin{align*}
\| \Xi_j (\phi_n)(t) \|_{H^s}
& \leq \| \jb{\xi}^s\|_{L^2_\xi(\supp \F[\Xi_j (\phi_n)])}
\| \Xi_j (\phi_n)(t) \|_{\F L^\infty} \notag\\
& \leq C^j t^{j} 
(RA^d)^{2j} R  f(A).
\end{align*}

\noi
On the other hand, 
by Lemma \ref{LEM:nonlin2} with \eqref{phi2} and \eqref{phi1a}, 
we have 
\begin{align}
\| \Xi_j (u_0 + \phi_n)(t)
 - \Xi_j ( \phi_n)(t)\|_{H^s}
& \leq \| \Xi_j (u_0 + \phi_n)(t)
- \Xi_j ( \phi_n)(t)\|_{L^2} \notag\\
& \leq C^j t^{j} 
\|u_0 \|_{L^2}
\big(
\| u_0\|_{\F L^1}^{2j}
+ \| \phi_n\|_{\F L^1}^{2j}\big)\notag\\
&  \les C^j t^{j} 
\| u_0\|_{L^2} (RA^d)^{2j}.
\label{Hs24}
\end{align}
	
\noi
Therefore, 
\eqref{Hs21}
follows from \eqref{Hs22} and \eqref{Hs24}.
\end{proof}

Next, we state a crucial proposition, 
establishing a lower bound on $\Xi_1(\phi_n)$.
This  proposition   played an important role
in establishing the norm inflation at 
the zero initial condition in \cite{Kishimoto}
and will also play an important role
in establishing norm inflation based at general initial data.
The proof in \cite{Kishimoto} exploits 
the high-to-low energy transfer mechanism in $\Xi_1(\phi_n)$.
We include the proof for readers' convenience.

\begin{proposition}\label{PROP:Hs3}
Let
$\phi_n$  be as in \eqref{phi1} and $s < 0$.
Then, for $0 < t \ll N^{-2}$, we have
\begin{align}
\|  \Xi_1 ( \phi_n)(t)\|_{H^s}
\ges t  R^3 A^{2d} \cdot f(A),
\label{Hs31}
\end{align}

\noi
where $f(A)$ is the function defined in \eqref{Hs21a}.
\end{proposition}

First,  recall the following simple lemma on the 
convolution of characteristic functions of cubes.

\begin{lemma}\label{LEM:conv}
Let $d \geq 1$.
Then, there exists $c_d$ such that 
\begin{align*}
\ind_{a + Q_A}* \ind_{b + Q_A} (\xi)\geq c_d A^d \ind_{a+b+Q_A}(\xi)
\end{align*}
		
\noi
for all $a, b, \xi \in \ft \M$ and $A \geq 1 $.
Here, $\ft \M$ denotes
 the Pontryagin dual of $\M$
 defined in \eqref{dual}.

\end{lemma}

We now present the proof of Proposition \ref{PROP:Hs3}.

\begin{proof}[Proof of Proposition \ref{PROP:Hs3}]

From \eqref{tree1}
with \eqref{tree1a}, we have
\begin{align*}
\Xi_1(\phi_n)(t)
= i 
\int_0^t S(t - t') [S(t')\phi_n\cj{S(t')\phi_n} S(t')\phi_n]dt'.
\end{align*}

\noi
Taking the Fourier transform, we have
\begin{align}
\F\big[\Xi_1(\phi_n)(t)\big](\xi)
= i e^{i |\xi|^2 t }
\intt_{\xi = \xi_1 - \xi_2 + \xi_3}\int_0^t & e^{- i t' (|\xi|^2 - |\xi_1|^2+|\xi_2|^2-|\xi_3|^2)}dt' 
\notag \\
& \ft \phi_n(\xi_1)
\cj{\ft \phi_n(\xi_2)} \ft\phi_n(\xi_3)
d\xi_1 d\xi_2.
\label{Hs32}
\end{align}

\noi
From \eqref{phi1}, 
we have $|\xi_j| \les N$ for $\xi_j \in \supp \ft \phi_n$.
Then, 
with
$\xi = \xi_1 - \xi_2 + \xi_3$, 
 we have
\begin{align}
|t' (|\xi|^2 - |\xi_1|^2+|\xi_2|^2-|\xi_3|^2)| \ll 1
\label{Hs33}
\end{align}

\noi
for $0< t ' \ll N^{-2}$.
Under the same condition, 
we have
\begin{align}
\Re \int_0^t  e^{- i t' (|\xi|^2 - |\xi_1|^2+|\xi_2|^2-|\xi_3|^2)}dt' 
\geq \frac 12 t. 
\label{Hs34}
\end{align}

\noi
Hence, it follows from \eqref{Hs32}, 
\eqref{Hs34}, 
and 
Lemma \ref{LEM:conv}
with \eqref{phi1}
that 
\begin{align}
|\F\big[\Xi_1(\phi_n)(t)\big](\xi)|
\ges t R^3 A^{2d} \cdot \ind_{Q_A}(\xi).
\label{Hs35}
\end{align}

\noi
Lastly, 
 noting that
$\| \jb{\xi}^s\|_{L^2_\xi( Q_A)}
\sim f(A)$,
we obtain 
\eqref{Hs31}.
\end{proof}

\subsection{Proof of Proposition \ref{PROP:main}}
\label{SUBSEC:3.2}

In this subsection, we prove Proposition \ref{PROP:main}.
We claim that it suffices to show that, 
given $n \in \N$, the following properties hold:
\begin{align*}
 \textup{(i)} & \quad R A^\frac{d}{2} N^s \ll \tfrac{1}{n}, \\ 
 \textup{(ii)} & \quad T R^2 A^{2d} \ll 1,  \\ 
 \textup{(iii)} & \quad T R^3 A^{2d} \cdot f(A) \gg n,\\ 
 \textup{(iv)} & \quad T R^3 A^{2d} \cdot f(A)\gg T^2 R^5 A^{4d} \cdot f(A), \\
 \textup{(v)} & \quad T \ll N^{-2}, \\
 \textup{(vi)} & \quad \eqref{phi1a}
 \quad \text{and}\quad  Rf(A) \gg \|u_0\|_{L^2}
 \end{align*}

\noi
for some $A, R, T$, and $N$, depending on $n$.
Here, $f(A)$ is as in \eqref{Hs21a}.
As before, implicit constants may depend on
(fixed) $u_0 \in \S(\M)$.

We first show how these conditions (i)-(v) 
imply Proposition \ref{PROP:main}.
The first condition (i) 
together with 
Lemma \ref{LEM:Hs1} 
verifies the first estimate in \eqref{main1}.

From \eqref{phi2} and \eqref{phi3} with \eqref{phi1a}, we have
\begin{align*}
\| u_{0, n} \|_{\F L^1} \sim R A^d.
\end{align*}
	
\noi	
Then, the second condition (ii) with Lemma \ref{LEM:LWP}
guarantees local existence of the solution $u_n$
on $[-T, T]$ with $u_n|_{t = 0} = u_{0, n}$
and Lemmas \ref{LEM:tree} and \ref{LEM:nonlin1} yield the convergence of the power series expansion \eqref{tree3}
in $C([-T, T]; \F L^1)$. 

Assuming the  conditions (ii) and (vi), %\eqref{cond2}, 
Lemma \ref{LEM:Hs2} yields
\begin{align}
 \bigg\|  \sum_{j = 2}^\infty \Xi_j (u_{0, n}) (T)\bigg\|_{H^s}
& \les T^2 R^4 A^{4d} \big\{ R  f(A) + \| u_0\|_{L^2}\big\} \notag\\
& \sim T^2 R^5 A^{4d} \cdot   f(A).
\label{Z1}
 \end{align}

\noi
Then, 
assuming the conditions (ii), (iii), (iv),  (v), and (vi), 
it follows from Lemma \ref{LEM:Hs1},  
Proposition \ref{PROP:Hs3}, and \eqref{Z1}
with the power series \eqref{tree3}
and \eqref{phi1a}
that 
\begin{align*}
\| u_n(T)\|_{H^s} & \geq 
\|  \Xi_1 ( \phi_n)(T)\|_{H^s}
- \| \Xi_0(u_{0,n})\|_{H^s} \notag \\
& \hphantom{X} - \| \Xi_1 (u_{0, n})(T) - \Xi_1 ( \phi_n)(T)\|_{H^s}
-\bigg\|  \sum_{j = 2}^\infty \Xi_j (u_{0, n}) (T)\bigg\|_{H^s} \notag\\
& \ges
T R^3 A^{2d}\cdot f(A)
- (1 + R A^\frac{d}{2} N^s) \notag\\
&\hphantom{XXXX}- 
TR^2 A^{2d} \|u_0\|_{L^2}- T^2 R^5 A^{4d} \cdot f(A)\notag\\
& \sim 
T R^3 A^{2d}\cdot f(A)
\gg n.
\end{align*}

\noi
This verifies the second estimate in \eqref{main1}
at time $t_n := T$.
Lastly, by choosing $N = N(n)$ sufficiently large, 
the  condition (v) guarantees that $t_n \in (0, \frac 1n)$.
This completes the proof of Proposition \ref{PROP:main}.	
	
\medskip

Therefore, it remains to 
verify the conditions (i)-(vi).
We divide the argument into the  following three cases:
(1) $s < -\frac d2$, 
(2) $s = -\frac d2$,
and (3) $  -\frac d2<s <0$.
Note that the last case is relevant only for $d \geq 2$.
Given $n \in \N$, we first 
choose appropriate $A, R,$ and  $T$ 
in terms of $N$.
Then, we choose
$N = N(n) \gg 1$
so that all the conditions (i)-(vi) are satisfied.
Note that the condition (iv) in fact 
follows from the condition (ii), 
and thus we only verify the condition (i), (ii), (iii), (v), and (vi)
in the following.

\medskip

\noi
$\bullet$
{\bf Case 1:} $s < - \frac  d 2$.
\quad 
In this case, we set
\begin{align}
A = N^{\frac 1 d (1 - \dl)}, 
\quad 
R = N^{2 \dl}, 
\quad \text{and} 
\quad
T = N^{ - 2 - 3 \dl}, 
\label{Z2}
\end{align}

\noi
where $\dl>0 $ is sufficiently small such that $s < - \frac 12 - \frac 32 \dl$.
The conditions (v) and (vi) are trivially satisfied for $N \gg 1$.
By choosing 
$N = N(n)$ sufficiently large, we have
\begin{align*}
R A^\frac{d}{2} N^s & = N^{s  + \frac 12 + \frac 32 \dl}\ll \tfrac 1n,
\\
TR^2 A^{2d} & = N^{-\dl} \ll 1, \\
TR^3 A^{2d} & = N^{\dl} \gg n, 
\end{align*}

\noi
verifying the conditions (i), (ii), and (iii), respectively.

\medskip

\noi
$\bullet$
{\bf Case 2:} $s  =  - \frac  d 2$.
\quad In this case, we set
\begin{align}
A = \frac{N^{\frac 1 d}}{(\log N )^{\frac 1{16d}}}, 
\quad R = 1, 
\quad \text{and} 
\quad
T = \frac 1{N^2 (\log N)^{\frac 18}} .
\label{Z2a}
\end{align}

\noi
As before, 
the conditions (v) and (vi) are trivially satisfied for $N \gg 1$.
By choosing 
$N = N(n)$ sufficiently large, we have
\begin{align*}
R A^\frac{d}{2} N^{-\frac d2} 
& = N^{\frac{1}{2}(1-d)}(\log N)^{-\frac{1}{32}} \ll \tfrac 1n, 
\\
TR^2 A^{2d} & = (\log N)^{-\frac 14} \ll 1, \\
TR^3 A^{2d} (\log A)^\frac{1}{2}& \sim
(\log N)^{-\frac 14} 
\big(\log N - \tfrac 1{16}\log \log N\big)^{\frac 12} 
\sim (\log N)^{\frac 14} 
 \gg n, 
\end{align*}

\noi
verifying the conditions (i), (ii), and (iii), respectively.

\medskip

\noi
$\bullet$
{\bf Case 3:} $ - \frac  d 2 < s < 0$.
\quad 
Recall that this case is relevant only for $d \geq 2$.
We set
\begin{align}
A = N^{\frac 2 d -\dl}, 
\quad 
R = N^{-1-s + \frac d 2\dl - \theta}, 
\quad \text{and} 
\quad
T= N^{ - 2 +2s + d \dl + \theta}, 
\label{Z3}
\end{align}

\noi
where $\dl \gg  \theta  > 0$ are sufficiently small
such that 
\begin{align} 
-2 s >d \dl + \theta
\qquad 
\text{and}
\qquad   -s \dl > 2 \theta.
\label{Z4}
\end{align}

\noi
It follows from \eqref{Z3} and \eqref{Z4}
that the conditions (v) and (vi) are  satisfied for $N \gg 1$.
Moreover, by choosing 
$N = N(n)$ sufficiently large, we have
\begin{align*}
R A^\frac{d}{2} N^s & = N^{-\theta}\ll \tfrac 1n,
\\
TR^2 A^{2d} & = N^{-\theta} \ll 1, \\
TR^3 A^{2d}\cdot A^{\frac d2 + s} & = N^{(\frac{-d + 2}{d})s - 2\theta - s\dl}
\geq N^{ - 2\theta - s\dl}
 \gg n, 
\end{align*}

\noi
verifying the conditions (i), (ii), and (iii), respectively.

\begin{remark}\label{REM:Wick2} \rm

In the following, we briefly discuss
necessary modifications
in proving Theorem \ref{THM:1}
for the Wick ordered NLS \eqref{NLS1} on $\T^d$.
We first define 
 the Duhamel integral  operator 
 $\wt \I$ 
 adapted to the renormalized nonlinearity in \eqref{NLS1}
 by setting
\begin{align*}
\wt \I[u_1, u_2, u_3](t)
= i 
\int_0^t S(t - t') \NN[u_1(t'), u_2(t'), u_3(t')] dt', 
\end{align*}

\noi
where $\NN$ is a trilinear operator given by 
\begin{align}
\F\big( \NN[f_1, f_2, f_3]\big)(\xi) 
= \sum_{\substack{\xi = \xi_1 - \xi_2 +  \xi_3\\ \xi \ne \xi_1, \xi_3}}
 \ft f_1(\xi_1) & \cj{\ft f_2(\xi_2)}\ft f_3(\xi_3) \notag\\
& -\ft f_1(\xi)\cj{\ft f_2(\xi)}\ft f_3(\xi).
\label{Z5}
\end{align}

\noi
Then, 
one can go through
Section \ref{SEC:2}
by replacing $\I$ with  $\wt \I$.
In particular, 
Lemmas \ref{LEM:nonlin1} and \ref{LEM:nonlin2} hold
with $\wt \I$ in place of $\I$.
This essentially follows from 
the following simple observation:
\begin{align*}
\big|\F\big( \NN[f_1, f_2, f_3]\big)(\xi) \big|
\leq
\sum_{\xi = \xi_1 - \xi_2 +  \xi_3}
|\ft f_1(\xi_1)||\ft f_2(\xi_2)||\ft f_3(\xi_3)|.
 \end{align*}

\noi
As a consequence, Lemmas \ref{LEM:Hs1} and \ref{LEM:Hs2}
also hold 
with $\wt \I$ in place of $\I$.
Lastly, let us consider Proposition \ref{PROP:Hs3}
in this case. 
Let $\phi_n$ be as in \eqref{phi1}, 
where $R = R(N)$ and $A = A(N)$ are as in 
\eqref{Z2}, \eqref{Z2a}, or \eqref{Z3}
with $N \gg 1$.
In particular, we have $1 \ll A \ll N$.
Then, 
it follows 
from \eqref{Z5} and \eqref{phi1}
that
\begin{align*}
\F\big( \NN[\phi_n, \phi_n, \phi_n]\big)(\xi) 
= \sum_{\xi = \xi_1 - \xi_2 +  \xi_3}
\ft \phi_n (\xi_1)\cj{\ft \phi_n(\xi_2)}\ft \phi_n(\xi_3)
= \F\big(|\phi_n|^2 \phi_n\big)(\xi)
\end{align*}

\noi
for $\xi \in Q_A$.
Hence, \eqref{Hs35} holds.
As a result, 
 Proposition \ref{PROP:Hs3}
also holds  for $\wt \I$.
The discussion in Subsection \ref{SUBSEC:3.2} holds
without any change.

\end{remark}

\begin{ackno}\rm
T.O.~was supported by the European Research Council (grant no.~637995 ``ProbDynDispEq'').
T.O.~is grateful to  Nobu Kishimoto
for  explaining  his work in \cite{Kishimoto}.
He would also like to thank Nikolay Tzvetkov
for a conversation on Remark \ref{REM:prob}.

\end{ackno}

\end{document}